\documentclass[twoside]{amsart} 
\usepackage{xcolor}
\usepackage{amssymb} 
\usepackage{amsthm}
\usepackage{amsmath}
\usepackage{mathtools} 
\usepackage{epigraph}
\usepackage{microtype}

\usepackage{hyperref} 
\definecolor{bluey}{rgb}{0,0,0.6}
\hypersetup{colorlinks,
linkcolor=bluey,
citecolor=bluey,
urlcolor=bluey}

\newtheorem{proposition}{Proposition}
\newtheorem{theorem}{Theorem}
\newtheorem{lemma}{Lemma} 
\newtheorem{corollary}{Corollary}

\makeatletter
\def\thmhead@plain#1#2#3{%
  \thmname{#1}\thmnumber{\@ifnotempty{#1}{ }\@upn{#2}}%
  \thmnote{ \the\thm@notefont(#3)}}
\let\thmhead\thmhead@plain
\makeatother

\begin{document}
\title{A generalization of Boppana's entropy inequality}
\author{Boon Suan Ho}
\address{Department of Mathematics, National University of Singapore}
\email{\href{mailto:hbs@u.nus.edu}{hbs@u.nus.edu}}

\begin{abstract} 
In recent progress on the union-closed sets conjecture, 
a key lemma has been Boppana's entropy inequality:
$h(x^2)\ge\phi xh(x)$, where $\phi=(1+\sqrt5)/2$ and $h(x)=-x\log x-(1-x)\log(1-x)$.
In this note, we prove that the generalized inequality
$\alpha_kh(x^k)\ge x^{k-1}h(x)$, first conjectured by Yuster, holds for real $k>1$, 
where $\alpha_k$ is the unique positive solution to $x(1+x)^{k-1}=1$. 
This implies an analogue of the union-closed sets conjecture
for approximate $k$-union closed set systems.
We also formalize our proof in Lean 4.
\end{abstract}

\maketitle

\section{Introduction}
The \emph{union-closed sets conjecture} states that for every
family $\{\varnothing\}\ne\mathcal F\subseteq2^{[n]}$ of sets closed
under unions, there exists $i\in[n]$ that is contained in $c=1/2$ of 
the sets of $\mathcal F$. The problem was posed in 1979, but no proof
for any $c>0$ was found until 2022, when Gilmer \cite{gil} gave an information-theoretic
proof for the $c=0.01$ case. His argument was quickly optimized to work for 
$c=(3-\sqrt5)/2\approx0.382$ by several authors, who all made use of \emph{Boppana's entropy inequality} \cite{bop} as a key ingredient:
\begin{proposition}
Let $\log$ denote the natural logarithm, and define the binary entropy function
$h(x)\coloneqq-x\log x-(1-x)\log(1-x)$ for $0<x<1$, setting $h(0)=h(1)=0$. 
Then $h(x^2)\ge\phi xh(x)$ for $0\le x\le1$, where $\phi=(1+\sqrt5)/2$,
and equality holds iff $x=0$, $1/\phi$, or $1$.
\end{proposition}
\noindent We refer the reader to \cite{cam} for further details and references.
In this note, we use standard calculus to prove the following generalization of Boppana's inequality:
\begin{theorem}
Let $k>1$ be real. Then $\alpha_kh(x^k)\ge x^{k-1}h(x)$ for $0\le x\le1$, where
$0<\alpha_k<1$ is the unique positive solution to $x(1+x)^{k-1}=1$,
and equality holds iff $x=0$, $1/(1+\alpha_k)$, or $1$.
\end{theorem}
\noindent Notice that Boppana's inequality is the case $k=2$.
Yuster \cite{yus1} conjectured this inequality for integer $k\ge2$, showing
that it implied a generalization of Gilmer's result to what he called
``approximate $k$-union closed set systems,'' and
proving it for $k=3$ and $k=4$. Consequently,
the following is a corollary of our result:

\begin{corollary}[\cite{yus1}, Conjecture 1.5]
Let $k\ge2$ be an integer and let $0\le c\le1$.
A finite set system $\mathcal F$ is 
$c$-approximate $k$-union closed if for at least
a $c$-fraction of the $k$-tuples $A_1,\dots,A_k\in\mathcal F$,
we have $\bigcup_{i=1}^kA_i\in\mathcal F$.

Let $\{\varnothing\}\ne\mathcal F\subseteq2^{[n]}$ be a
$(1-\epsilon)$-approximate $k$-union closed set system,
where $0\le\epsilon<1/2$. Then there exists an element
contained in an $\alpha_k/(1+\alpha_k)-\delta$ fraction of sets in $\mathcal F$,
where $\delta=(k\epsilon+2\epsilon\log(1/\epsilon)/\log|\mathcal F|)^{1/(k-1)}$.
\end{corollary}

Later, Yuster and Yashfe \cite{yus2}
proved the inequality for integer $5\le k\le20$.
Wakhare \cite{wak} investigated the generalization to real $k>1$.

\section{The proof}
Fix real $k>1$ and write $\alpha=\alpha_k$. 
Define
$q(x)\coloneqq x^{k-1}h(x)/h(x^k)$ on $(0,1)$, 
and extend $q$ to $0$ and $1$ by taking limits,
so that $q(0)=q(1)=1/k$ (\hyperref[lemma:l1]{Lemma 1}).
Our goal is to show that $q(x)\le\alpha$.
Since $\alpha=1/(1+\alpha)^{k-1}$, we have 
$$q\Bigl({1\over1+\alpha}\Bigr)
={1\over(1+\alpha)^{k-1}}\cdot
{h\bigl({1\over1+\alpha}\bigr)\over 
h\bigl({1\over(1+\alpha)^k}\bigr)}
=\alpha\cdot{h\bigl({1\over1+\alpha}\bigr)\over 
h\bigl({\alpha\over1+\alpha}\bigr)}=\alpha,$$
where we used the fact that $h(x)=h(1-x)$ in the last step.

We will complete the proof by showing that $q'(x)=0$ if and only if $x=1/(1+\alpha)$;
since $q$ is differentiable on $(0,1)$, continuous on $[0,1]$, and $q(1/(1+\alpha))=\alpha>1/k=q(0)=q(1)$ (\hyperref[lemma:l2]{Lemma 2}), this will imply that $q(x)\le\alpha$ as needed (\hyperref[lemma:l3]{Lemma 3}).

Suppose $q'(x)=0$. 
Since
\begin{equation}
q'(x)={(k-1)x^{k-2}h(x)+x^{k-1}h'(x)\over h(x^k)}-x^{k-1}h(x){h'(x^k)kx^{k-1}\over h(x^k)^2},
\end{equation}
the condition $q'(x)=0$ is equivalent (after multiplying by $h(x^k)^2/x^{k-2}$) to
\begin{equation} \label{eq:t}
\Bigl((k-1)h(x)+xh'(x)\Bigr)h(x^k)=kx^kh(x)h'(x^k).
\end{equation}
Since $h'(x)=\log(1-x)-\log x$, we have $xh'(x)-h(x)=\log(1-x)$,
so 
\begin{equation}
xh'(x)=h(x)+\log(1-x)\quad\text{and}\quad x^kh'(x^k)=h(x^k)+\log(1-x^k).
\end{equation}
Substituting into \eqref{eq:t} then yields
\begin{equation}
\Bigl(kh(x)+\log(1-x)\Bigr)h(x^k)=kh(x)\Bigl(h(x^k)+\log(1-x^k)\Bigr),
\end{equation}
or
\begin{equation}
\log(1-x)h(x^k)=kh(x)\log(1-x^k).
\end{equation}
Multiplying by $\log x$ and dividing by $h(x)h(x^k)$ then gives
\begin{equation}
{\log(x)\log(1-x)\over h(x)}={k\log(x)\log(1-x^k)\over h(x^k)},
\end{equation}
and since $\log(x^k)=k\log x$, this equation becomes 
\begin{equation} \label{eq:main}
U(x)=U(x^k),
\end{equation}
where
\begin{equation}
U(x)={\log(x)\log(1-x)\over h(x)}.
\end{equation} 
Since $U(x)=U(1-x)$, and
since $U$ is strictly decreasing on $(0,1/2]$ (\hyperref[lemma:l4]{Lemma 4}),
it follows that every value of $U$ is attained at exactly
two points $x$ and $1-x$ (except at $x=1/2$, where they coincide).
Thus $U(x)=U(x^k)$ implies $x^k=x$ or $x^k=1-x$. Since $k>1$
and $0<x<1$, it follows that $x^k=1-x$, which in turn implies
that $x=1/(1+\alpha)$ (\hyperref[lemma:l5]{Lemma 5}). This completes the proof.

\section{Lemmas}
\begin{lemma} \label{lemma:l1}
We have $\lim_{x\to 0^+}q(x)=\lim_{x\to1^-}q(x)=1/k$.
\end{lemma}

\begin{proof}
The $x\to0^+$ case follows from making the asymptotic estimate
$h(x)=-x\log x(1+o(1))$, since then
$h(x^k)=-x^k\log(x^k)(1+o(1))=-kx^k\log x(1+o(1))$ and
\begin{equation}
q(x)={x^{k-1}h(x)\over h(x^k)}={x^{k-1}x\log x(1+o(1))\over kx^k\log x(1+o(1))}\to{1\over k};
\end{equation}
a similar approach works for $x\to1^-$ by using the symmetry $h(1-x)=h(x)$.
\end{proof}

\begin{lemma} \label{lemma:l2}
We have $\alpha>1/k$.
\end{lemma}

\begin{proof}
Recall that $\alpha(1+\alpha)^{k-1}=1$. Set $f(x)=x(1+x)^{k-1}$.
Then $f$ is strictly increasing on $(0,\infty)$, as can be seen
by considering $f'$. Thus it suffices to check that $f(1/k)=(1/k)(1+1/k)^{k-1}<1$,
or $(k+1)^{k-1}<k^k$. This follows from taking logarithms and doing calculus,
or alternatively from the weighted AM-GM inequality with numbers $x_1=1$, $x_2=k+1$
and weights $w_1=1$, $w_2=k-1$.
\end{proof}

\begin{lemma} \label{lemma:l3}
Given $M>0$, $0<a<1$, and a continuous function $f\colon[0,1]\to\mathbf R$
that is differentiable on $(0,1)$, suppose that $f(0)<M$, $f(a)=M$, $f(1)<M$, and
$f'(x)=0$ iff $x=a$. Then $f(x)\le M$ with equality iff $x=a$. 
\end{lemma}

\begin{proof}
This follows from the extreme value theorem and Fermat's theorem. 
\end{proof}

\begin{lemma} \label{lemma:l4}
Let $U(x)=\log(x)\log(1-x)/h(x)$. Then $U$ is decreasing on $(0,1/2]$.
\end{lemma}

\begin{proof}
It suffices to prove that $1/U$ is increasing on $(0,1/2)$.
Since \begin{equation}
{1\over U(x)}={-x\over\log(1-x)}+{-(1-x)\over\log x}=L(1,1-x)+L(1,x)
\end{equation}
where $L(a,b)\coloneqq(a-b)/(\log a-\log b)$ is the \emph{logarithmic mean},
the identity $L(a,b)=\int_0^1a^{1-s}b^s\,ds$ yields
$f(t)\coloneqq L(1,t)=\int_0^1t^s\,ds$. 
Then $f''(t)=\int_0^1s(s-1)t^{s-2}\,ds<0$ on $(0,1)$, 
so $f(t)$ is strictly concave.
Thus $f'(t)$ is strictly decreasing, and we conclude that 
$(1/U(x))'=f'(x)-f'(1-x)>0$ for $x\in(0,1/2)$ as needed.
\end{proof}

\begin{lemma} \label{lemma:l5}
Suppose $x^k=1-x$. Then $x=1/(1+\alpha)$, where $\alpha$ is the unique
positive solution to $\alpha(1+\alpha)^{k-1}=1$.
\end{lemma}

\begin{proof}
Substitute $x=1/(1+\alpha)$ into $x^k=1-x$ to get
$1/(1+\alpha)^k=\alpha/(1+\alpha)$, then multiply both sides
by $(1+\alpha)^k$.
\end{proof}

\section{Remarks}
Though this paper was written and checked carefully by hand,
key steps in some proofs were generated with the assistance of GPT-5.2 pro.
The result has also been formalized in Lean 4 using Harmonic Aristotle
and Gemini 3 Pro Preview; the code is available from 
\url{https://github.com/boonsuan/entropy-inequality}.

\section{Acknowledgements}
The author thanks Hao Huang for helpful comments on a draft of this note.

\end{document}